\documentclass[sn-mathphys]{sn-jnl}

\jyear{2021}%

\theoremstyle{thmstyleone}%
\newtheorem{theorem}{Theorem}

\theoremstyle{thmstyletwo}%
\newtheorem{remark}{Remark}
\newtheorem{corollary}{Corollary}[section]
\newtheorem{lemma}[theorem]{Lemma}

\theoremstyle{thmstylethree}%
\usepackage{amsmath}
\newtheorem{dfn}{Definition}
\usepackage{url}
\begin{document}
	
	\title[Ramanujan Theta Function Identities and Quadratic Numbers ]{\textsc{Ramanujan Theta Function Identities and Quadratic Numbers}\begin{flushleft}
			\begin{center}
				
			\end{center}
	\end{flushleft}}

	\author*[1,2]{\fnm{Hemant} \sur{Masal}}\email{hemantmasal@gmail.com}
	
	\author[2]{\fnm{Hemant} \sur{Bhate}}\email{bhatehemant@gmail.com}

	\author[2]{\fnm{Subhash} \sur{Kendre}}\email{sdkendre@yahoo.com}

	\affil*[1]{\orgdiv{Department of First Year Engineering}, \orgname{Pune Institute of Computer Technology}, \orgaddress{\street{Dhankawadi}, \city{Pune}, \postcode{411 043}, \state{Maharashtra}, \country{India}}}
	
	\affil[2]{\orgdiv{Department of Mathematics}, \orgname{Savitribai Phule Pune University}, \orgaddress{\street{Ganeshkhind}, \city{Pune}, \postcode{411 007}, \state{Maharashtra}, \country{India}}}

	\abstract {Eigenvectors of the discrete Fourier transform can be expressed using Ramanujan theta functions. New theta function identities, Ramanujan theta function identities, and generating functions for the quadratic numbers are a consequence.}

	
	
	\maketitle

\section{Introduction}

Let $A$ denote the matrix of the discrete Fourier transform. The $(j,k)^{th}$ entry of the DFT matrix $A$ of size $n$ is $A_{jk}=\frac{1}{\sqrt{n}}e^{\frac{2\pi i jk}{n}}, ~~j,k \in \mathbb{Z}_n$. It is clear that $A^4=I$, and hence the eigenvalues are $1, -1,i,-i$ with non negative multiplicities   $1, -1,i,-i$ are $[(n+4)/4], [(n+2)/4], [(n+1)/4], [(n-1)/4]$ respectively \cite{matveev2001intertwining,mehta1987eigenvalues}, where $[x]$ denotes greatest integer not greater than $x.$

Recall following Theorem from  Matveev \cite{matveev2001intertwining}:

\begin{theorem}\label{M1}
	Let $\sum_{m \in \mathbb{Z}}g_{m}$ be any absolutely convergent series. Then the vector whose $j^{th}$ component is 
	\begin{equation*}
		v_{j}(k)=\sum_{m \in \mathbb{Z}} (g_{mn+j}+(-1)^k g_{mn-j})+ \frac{1}{\sqrt{n}}  \sum_{m \in \mathbb{Z}} \left[(-i)^k g_{m} + (i)^{k}g_{-m} \right]e^{\frac{2\pi i mj}{n}}
	\end{equation*}
	is an eigenvector of DFT of order $n$ corresponding to eigenvalue $i^k$.
\end{theorem}

\section{The Ramanujan theta function}
\begin{dfn}
	The Ramanujan theta function is defined by \cite{berndt2012ramanujan,Bruce,Robert}, 
	
	\begin{equation}
		f(a,b)=\sum_{m \in \mathbb{Z}}a^{\frac{m(m+1)}{2}}b^{\frac{m(m-1)}{2}},~~~~~~ \mid ab \mid<1.
	\end{equation}
\end{dfn}
The Ramanujan theta function satisfy \cite{berndt2012ramanujan}

\begin{equation*}
	f(a,b)=f(b,a),f(1,a)=2f(a,a^3),f(-1,a)=0.
\end{equation*}

Also, for fixed integer $n,$

\begin{equation*}
	f(a,b)=a^{\frac{n(n+1)}{2}}b^{\frac{n(n-1)}{2}} f(a(ab)^n, b(ab)^{-n}).
\end{equation*}

Recall that the notation for the infinite product is, 
\begin{equation*}
	(\alpha:\beta)_{\infty}=\prod_{k=0}^{\infty}(1-\alpha \beta ^k),~~\text{and} ~(\alpha, \gamma, \cdots, \zeta:\beta)_{\infty}=
	(\alpha:\beta)_{\infty}(\gamma:\beta)_{\infty} \cdots (\zeta:\beta)_{\infty} 
\end{equation*}

The well known Jacobi triple product identity \cite{berndt2012ramanujan,murty2015problems} gives us following 
\begin{equation}\label{j1}
	\begin{split}
		f(a,b) = (-a:ab)_{\infty} (-b:ab)_{\infty}(ab:ab)_{\infty}
		= (-a, -b,-ab : ab)_{\infty}.
	\end{split}
\end{equation}

We now introduce the Ramanujan theta function with characteristics $(\alpha, \beta , c)$ as 

\begin{equation}
	f_{(\alpha, \beta , c)}(a,b,x)=\sum_{m \in \mathbb{Z}}a^{\frac{(m+\alpha)(m+\alpha+1/c)}{2}}b^{\frac{(m+\alpha)(m+\alpha-1/c)}{2 }}e^{2\pi i (m+\alpha)(x+\beta)}.
\end{equation}
For $\alpha =0, \beta =0, x=0, c=1$ we will get Ramanujan theta function. 

\section{ The Eigenvectors of DFT} 
The eigenvalue and eigenvectors decomposition of DFT is studied by McClellan and Parks \cite{mcclellan1972eigenvalue}. Mehta \cite{mehta1987eigenvalues} studied eigenvctors of DFT using Hermite functions whereas Matveev \cite{matveev2001intertwining} proved that Jacobi theta functions also gives eigenvectors of the DFT. Following Theorem gives the eigenvectors of the DFT in the form of Ramanujan theta function.
\begin{theorem}\label{RT1}
	For $n$ fixed,  $a,b \in \mathbb{C}$ with
	$\mid ab\mid <1$ and for any $x \in \mathbb{C}$, the $j^{th},$ component with $j\in \{ 0,1,2, \cdots, n-1\},$ of the eigenvector corresponding to eigenvalue $i^k$ of the DFT is 
	
	\begin{equation*}
		\begin{split}
			v_j(k)=&f_{(\frac{j}{n},0,n)} (a,b,x)+ (-1)^k f_{(\frac{-j}{n},0,n)} (a,b,x) \\
			+&\frac{1}{\sqrt{n}} \left[ (-i)^k f_{(0,\frac{j}{n},1)} (a^{1/n^2},b^{1/n^2},x/n)+ (-i)^{3k} f_{(0,\frac{-j}{n},1)} (a^{1/n^2},b^{1/n^2},x/n)\right].    
		\end{split}
	\end{equation*}
\end{theorem}	
\begin{remark}
	This follows by taking
	\begin{equation}
		g_m=a^{\frac{m(m+1)}{2n^2}}b^{\frac{m(m-1)}{2n^2}} e^{2\pi i \frac{m}{n}x}
	\end{equation}
	and applying Theorem \ref{M1}. 
\end{remark}
\section{Functional Identities}	
In this section we established some functional relations between Ramanujan theta functions.
\begin{lemma}\label{T1}
	The following functional equations holds for the DFT of size $2$
	\begin{enumerate}
		\item 
		$
			2 f_{(0,0,2)}(a,b,x)+f_{(\frac{1}{2},0,2)}(a,b,x)+f_{(\frac{-1}{2},0,2)}(a,b,x)= 2 f_{(0,0,1)} (a^{1/4}, b^{1/4},x/2)
	$
		\item
		$
			2 f_{(0,0,2)}(a,b,x)-[f_{(\frac{1}{2},0,2)}(a,b,x)+f_{(\frac{-1}{2},0,2)}(a,b,x)]
			= f_{(0,\frac{1}{2},1)}(a^{1/4}, b^{1/4},x/2)+ f_{(0,\frac{-1}{2},1)}(a^{1/4}, b^{1/4},x/2).
			$ 
	\end{enumerate}
\end{lemma}
\begin{proof}
	The DFT $A$ of size $2$ has two eigenvalues $1, -1$. Both the eigenvalues are non degenerate. Let $v_1, v_2$ be the eigenvectors corresponding to the  eigenvalues $1, -1$ respectively. By using Theorem \ref{RT1}, we have 
	
	\begin{equation*}
		v_1=\begin{pmatrix}
			v_{0}(0)\\
			v_{1}(0)
		\end{pmatrix}~~ \text{and}~~
		v_2=\begin{pmatrix}
			v_{0}(2)\\
			v_{1}(2)
		\end{pmatrix},
	\end{equation*}
	where, 
	
	\begin{equation*}
		\begin{split}
			v_{0}(0)=&2 f_{(0,0,2)}(a,b,x)+\sqrt{2}f_{(0,0,1)} (a^{1/4}, b^{1/4},x/2)\\
			v_{1}(0)=& f_{(\frac{1}{2},0,2)}(a,b,x)+f_{(\frac{-1}{2},0,2)}(a,b,x)\\
			&+\frac{1}{\sqrt{2}}\left[f_{(0,\frac{1}{2},1)}(a^{1/4}, b^{1/4},x/2)+ f_{(0,\frac{-1}{2},1)}(a^{1/4}, b^{1/4},x/2)\right],\\
			v_{0}(2)=&2 f_{(0,0,2)}(a,b,x)-\sqrt{2}f_{(0,0,1)} (a^{1/4}, b^{1/4},x/2),\\
			v_{1}(2)=& f_{(\frac{1}{2},0,2)}(a,b,x)+f_{(\frac{-1}{2},0,2)}(a,b,x)\\
			&-\frac{1}{\sqrt{2}}\left[f_{(0,\frac{1}{2},1)}(a^{1/4}, b^{1/4},x/2)+ f_{(0,\frac{-1}{2},1)}(a^{1/4}, b^{1/4},x/2)\right].\\
		\end{split}
	\end{equation*}
	So,
	\begin{align}\label{r1}
		v_1+v_2=\begin{pmatrix}
			4 f_{(0,0,2)}(a,b,x)\\
			2\left[f_{(\frac{1}{2},0,2)}(a,b,x)+f_{(\frac{-1}{2},0,2)}(a,b,x)\right]
		\end{pmatrix}.
	\end{align}
	and 
	\begin{align}\label{r2}
		v_1-v_2=\begin{pmatrix}
			2\sqrt{2}f_{(0,0,1)} (a^{1/4}, b^{1/4},x/2)\\
			\sqrt{2}\left[f_{(0,\frac{1}{2},1)}(a^{1/4}, b^{1/4},x/2)+ f_{(0,\frac{-1}{2},1)}(a^{1/4}, b^{1/4},x/2)\right]
		\end{pmatrix}.
	\end{align}
	It is clear that
	\begin{equation}\label{r3}
		A(v_1+v_2)=v_1-v_2.
	\end{equation}
	
	Equations (\ref{r1}),(\ref{r2}) and (\ref{r3}) give the required results.
\end{proof}

The immediate consequence of the lemma above is the next corollary.
\begin{corollary} \label{r4} The following equations hold.
	\begin{enumerate}
		\item 
		$4 f_{(0,0,2)}(a,b,x)
		= 2 f_{(0,0,1)} (a^{1/4}, b^{1/4},x/2)
		+f_{(0,\frac{1}{2},1)}(a^{1/4}, b^{1/4},x/2)+ f_{(0,\frac{-1}{2},1)}(a^{1/4}, b^{1/4},x/2),$
		\item
		$2\left[f_{(\frac{1}{2},0,2)}(a,b,x)+f_{(\frac{-1}{2},0,2)}(a,b,x)\right]
		=2 f_{(0,0,1)} (a^{1/4}, b^{1/4},x/2)
		-[f_{(0,\frac{1}{2},1)}(a^{1/4}, b^{1/4},x/2)+ f_{(0,\frac{-1}{2},1)}(a^{1/4}, b^{1/4},x/2)].$
		\item The quadratic identity:  This identity can be considers as product identity for Ramanujan theta functions \cite{berndt2012ramanujan}. 
		\begin{equation*}
			\begin{split}
				&4 f^2_{(0,0,2)}(a,b,x)-\left[f_{(\frac{1}{2},0,2)}(a,b,x)+f_{(\frac{-1}{2},0,2)}(a,b,x) \right]^2\\
				&=2 f_{(0,0,1)} (a^{1/4}, b^{1/4},x/2)\left[ f_{(0,\frac{1}{2},1)}(a^{1/4}, b^{1/4},x/2)+ f_{(0,\frac{-1}{2},1)}(a^{1/4}, b^{1/4},x/2)\right].
			\end{split}
		\end{equation*}
	\end{enumerate} 	
\end{corollary}

We now derive the identity corresponding to the DFT of size $3.$

\begin{lemma}\label{r5}
	\begin{equation}
		\begin{split}
			& \frac{2\sqrt{3}}{1+\sqrt{3}}\sum_{m=0}^{\infty} a^{\frac{3m(3m+1)}{2}}b^{\frac{3m(3m-1)}{2}}e^{2\pi imx}\\
			-\sqrt{3} &\left[\sum_{m=0}^{\infty} a^{\frac{(3m+1)(3m+2)}{2}}b^{\frac{3m(3m+1)}{2}}e^{2\pi i(m+1/3)x}
			+ \sum_{m=0}^{\infty} a^{\frac{3m(3m-1)}{2}}b^{\frac{(3m-1)(3m-2)}{2}}e^{2\pi i(m-1/3)x}\right]\\
			= &\left(ab,-ae^{\frac{2\pi i(x+1)}{3}}, -be^{\frac{-2\pi i(x+1)}{3}} : ab\right)_{\infty} + \left(ab,-ae^{\frac{2\pi i(x-1)}{3}}, -be^{\frac{-2\pi i(x-1)}{3}} : ab\right)_{\infty} \\
			&- \frac{-2}{1+\sqrt{3}} \left(ab,-ae^{2\pi ix}, -be^{-2\pi ix} : ab\right)_{\infty}
		\end{split}
	\end{equation}
\end{lemma}

\begin{proof}
	The DFT of size $n=3$ has eigenvalues $1,-1,i$ and all non-degenerate. The eigenvectors corresponding to eigenvalue $1$ are,
	\begin{equation*}
		v_1= \begin{bmatrix}
			1+\sqrt{3}\\
			1\\
			1
		\end{bmatrix}, 
	\end{equation*}
	Let,
	\begin{equation*}
		\begin{split}
			v_{0}(0)=&2f_{(0,0,3)}(a,b,x) + \frac{2}{\sqrt{3}}f_{(0,0,1))}(a^{1/9},b^{1/9},x/3),\\
			v_{1}(0)=&f_{(\frac{1}{3},0,3)}(a,b,x) + f_{(\frac{-1}{3},0,3)}(a,b,x)\\
			&+ \frac{1}{\sqrt{3}} \left[ f_{(0,\frac{1}{3},1)}(a^{1/9},b^{1/9},x/3) + f_{(0,\frac{-1}{3},1)}(a^{1/9},b^{1/9},x/3)\right],\\
			v_{2}(0)=&	f_{(\frac{2}{3},0,3)}(a,b,x) + f_{(\frac{-2}{3},0,3)}(a,b,x)\\
			&+ \frac{1}{\sqrt{3}} \left[ f_{(0,\frac{2}{3},1)}(a^{1/9},b^{1/9},x/3) + f_{(0,\frac{-2}{3},1)}(a^{1/9},b^{1/9},x/3)\right]
		\end{split}
	\end{equation*}
	\begin{equation*}
		v_2= \begin{bmatrix}
			v_{0}(0)\\
			v_{1}(0)\\
			v_{2}(0)
		\end{bmatrix}.
	\end{equation*}
	It is clear that any $2\times 2$ minor of $[v_1, v_2]$ vanishes. So,
	\begin{equation*}
		\det \begin{bmatrix}
			1+\sqrt{3} &  v_{0}(0)\\
			1 & v_{1}(0)
		\end{bmatrix}=0
	\end{equation*}
	\begin{equation}
		\begin{split}
			&2f_{(0,0,3)}(a,b,x) + \frac{2}{\sqrt{3}}f_{(0,0,1))}(a^{1/9},b^{1/9},x/3) \\
			=& (1+\sqrt{3}) \left[f_{(\frac{1}{3},0,3)}(a,b,x) + f_{(\frac{-1}{3},0,3)}(a,b,x)\right]\\
			&+ \frac{1+\sqrt{3}}{\sqrt{3}} \left[ f_{(0,\frac{1}{3},1)}(a^{1/9},b^{1/9},x/3) + f_{(0,\frac{-1}{3},1)}(a^{1/9},b^{1/9},x/3)\right]
		\end{split}   
	\end{equation}
	i.e.
	
	\begin{equation}\label{R2:5}
		\begin{split}
			&\frac{2\sqrt{3}}{1+\sqrt
				{3}} f_{(0,0,3)}(a,b,x) -\sqrt{3} \left[ f_{(\frac{1}{3},0,3)}(a,b,x) + f_{(\frac{-1}{3},0,3)}(a,b,x)\right]\\
			=&  f_{(0,\frac{1}{3},1)}(a^{1/9},b^{1/9},x/3) 
			+f_{(0,\frac{-1}{3},1)}(a^{1/9},b^{1/9},x/3)
			-\frac{2}{1+\sqrt{3}}f_{(0,0,1))}(a^{1/9},b^{1/9},x/3) 
		\end{split}
	\end{equation}
	
	This leads to,
	
	\begin{equation}\label{3:1}
		\begin{split}
			& \frac{2\sqrt{3}}{1+\sqrt{3}}\sum_{m=0}^{\infty} a^{\frac{m(m+1/3)}{2}}b^{\frac{m(m-1/3)}{2}}e^{2\pi imx}\\ &-\sqrt{3} \sum_{m=0}^{\infty} a^{\frac{(m+1/3)(m+2/3)}{2}}b^{\frac{m(m+1/3)}{2}}e^{2\pi i(m+1/3)x}\\
			&-\sqrt{3}
			\sum_{m=0}^{\infty} a^{\frac{m(m-1/3)}{2}}b^{\frac{(m-1/3)(m-2/3)}{2}}e^{2\pi i(m-1/3)x}\\
			= & \sum_{m=0}^{\infty} a^{\frac{m(m+1)}{18}}b^{\frac{m(m-1)}{18}}e^{2\pi im(x+1)/3}  + \sum_{m=0}^{\infty} a^{\frac{m(m+1)}{18}}b^{\frac{m(m-1)}{18}}e^{2\pi im(x-1)/3} \\
			&- \frac{2}{1+\sqrt{3}} \sum_{m=0}^{\infty} a^{\frac{m(m+1)}{18}}b^{\frac{m(m-1)}{18}}e^{2\pi imx}
		\end{split}
	\end{equation}
	
	The right hand side of the equation (\ref{3:1}) can be expressed as follows, 
	
	\begin{equation*}
		\begin{split}
			f(a^{(1/9)}e^{2\pi ix}, b^{(1/9)}e^{-2\pi ix})=&\sum_{m=0}^{\infty} (a^{(1/9)}e^{2\pi ix})^{m(m+1)/2}(b^{(1/9)}e^{-2\pi ix})^{m(m-1)/2}\\
			&= \sum_{m=0}^{\infty} a^{\frac{m(m+1)}{18}}b^{\frac{m(m-1)}{18}}e^{2\pi imx}.
		\end{split}
	\end{equation*}
	The equation (\ref{3:1}) becomes, 
	
	\begin{equation}\label{3:2}
		\begin{split}
			& \frac{2\sqrt{3}}{1+\sqrt{3}}\sum_{m=0}^{\infty} a^{\frac{m(m+1/3)}{2}}b^{\frac{m(m-1/3)}{2}}e^{2\pi imx}\\ &-\sqrt{3} \sum_{m=0}^{\infty} a^{\frac{(m+1/3)(m+2/3)}{2}}b^{\frac{m(m+1/3)}{2}}e^{2\pi i(m+1/3)x}\\
			&-\sqrt{3} \sum_{m=0}^{\infty} a^{\frac{m(m-1/3)}{2}}b^{\frac{(m-1/3)(m-2/3)}{2}}e^{2\pi i(m-1/3)x}\\
			= & f\left(a^{(1/9)}e^{2\pi i(x+1)/3}, b^{(1/9)}e^{-2\pi i(x+1)/3}\right) + f\left(a^{(1/9)}e^{2\pi i(x-1)/3}, b^{(1/9)}e^{-2\pi i(x-1)/3}\right)\\
			& - \frac{2}{1+\sqrt{3}}f\left(a^{(1/9)}e^{2\pi ix}, b^{(1/9)}e^{-2\pi ix}\right).
		\end{split}
	\end{equation}
\end{proof}

The Jacobi theta function identities can be derived from these identities of Ramanujan theta function.
\begin{lemma}
	For any function $f$  from the upper half plane to itself following identities hold.
	\begin{enumerate}
		\item  
		$2\theta \left( \frac{f(\tau)-\tau}{4}+x,\tau \right)-2\theta \left( \frac{f(\tau)-\tau}{8}+\frac{x+1}{2}, \frac{\tau}{4} \right)\\
		=e^{\pi i(\frac{f(\tau)}{4}+x)}\theta \left(\frac{f(\tau)+\tau}{4}+x, \tau \right)+ e^{\pi i(\frac{2\tau-f(\tau)}{4}-x)}\theta \left( \frac{f(\tau)-3\tau}{4}+x, \tau \right),$
		\item 
		$\frac{2\sqrt{3}}{1+\sqrt{3}}\theta \left(\frac{f(\tau)-\tau}{6}+x, \tau\right)\\
		-\sqrt{3}\left[e^{\frac{\pi i}{3}\left(\frac{f(\tau)}{3}+2x\right)}\theta \left(\frac{f(\tau)+\tau}{6}+x, \tau\right)+ e^{\frac{\pi i}{3}\left(\frac{2\tau-f(\tau)}{3}-2x\right)}\theta \left(\frac{f(\tau)-3\tau}{6}+x, \tau\right)\right]\\
		=\theta \left(\frac{f(\tau)-\tau}{18}+\frac{x+1}{3}, \frac{\tau}{9}\right)+\theta \left(\frac{f(\tau)-\tau}{18}+\frac{x-1}{3}, \frac{\tau}{9}\right)-\frac{2}{1+\sqrt{3}}\theta \left(\frac{f(\tau)-\tau}{18}+\frac{x}{3}, \frac{\tau}{9}\right).$
	\end{enumerate}
\end{lemma}

\begin{proof}
	\begin{enumerate}
		\item 
		Let $a=e^{\pi i f(\tau)}, b=e^{\pi ig(\tau)}$ where $f, g$ are the functions from upper half of complex plane to itself. 
		\begin{equation}\label{R2:1}
			\begin{split}
				f_{(0,0,2)}(a,b,x)=& \sum_{m \in \mathbb{Z}}e^{\pi i f(\tau)\left( \frac{m^2}{2}+\frac{m}{4} \right)}e^{\pi i g(\tau)\left( \frac{m^2}{2}-\frac{m}{4} \right)}e^{2\pi imx}\\
				=& \sum_{m \in \mathbb{Z}} e^{\pi i \left( \frac{(f+g)(\tau)}{2} \right)m^2}e^{2\pi im \left(\frac{(f-g)(\tau)}{8}+x \right)}\\
				=& \theta \left(\frac{(f-g)(\tau)}{8}+x, \frac{(f+g)(\tau)}{2} \right),
			\end{split}
		\end{equation}
		
		\begin{equation}\label{R2:2}
			\begin{split}
				f_{(\frac{1}{2},0,2)}(a,b,x)=&\sum_{m \in \mathbb{Z}}e^{\pi i f(\tau)\left( \frac{m^2}{2}+\frac{3m}{4}+\frac{1}{4} \right)}e^{\pi i g(\tau)\left( \frac{m^2}{2}+\frac{m}{4} \right)}e^{2\pi i(m+1/2)x} \\
				=& e^{\pi i(\frac{f(\tau)}{4}+x)}\sum_{m \in \mathbb{Z}} e^{\pi i \left( \frac{(f+g)(\tau)}{2} \right)m^2}e^{2\pi im \left(\frac{(3f+g)(\tau)}{8}+x \right)}\\
				=& e^{\pi i(\frac{f(\tau)}{4}+x)} \theta \left(\frac{(3f+g)(\tau)}{8}+x, \frac{(f+g)(\tau)}{2} \right),
			\end{split}
		\end{equation}
		
		\begin{equation}\label{R2:3}
			\begin{split}
				f_{(\frac{-1}{2},0,2)}(a,b,x)=& e^{\pi i(\frac{g(\tau)}{4}-x)}\sum_{m \in \mathbb{Z}} e^{\pi i \left( \frac{(f+g)(\tau)}{2} \right)m^2}e^{2\pi im \left(\frac{-(f+3g)(\tau)}{8}+x \right)}\\
				=& e^{\pi i(\frac{g(\tau)}{4}-x)} \theta \left( \frac{-(f+3g)(\tau)}{8}+x, \frac{(f+g)(\tau)}{2} \right),
			\end{split}
		\end{equation}
		
		\begin{equation} \label{R2:4}
			\begin{split}
				f_{(0,1/2,1)}(a^{1/4}, b^{1/4},x/2)= & \sum_{m \in \mathbb{Z}} e^{\pi i \left( \frac{(f+g)(\tau)}{8} \right)m^2}e^{2\pi im \left(\frac{(f-g)(\tau)}{16}+\frac{x+1}{2} \right)}\\
				=& \theta\left(\frac{(f-g)(\tau)}{16}+\frac{x+1}{2}, \frac{(f+g)(\tau)}{8} \right)
			\end{split}
		\end{equation}
		
		Using equations (\ref{R2:1}), (\ref{R2:2}), (\ref{R2:3}) and (\ref{R2:4}) in the second identity of lemma \ref{T1} we have,
		\begin{equation}
			\begin{split}
				&2\theta \left(\frac{(f-g)(\tau)}{8}+x, \frac{(f+g)(\tau)}{2} \right)-2\theta\left(\frac{(f-g)(\tau)}{16}+\frac{x+1}{2}, \frac{(f+g)(\tau)}{8} \right)\\
				=& e^{\pi i(\frac{f(\tau)}{4}+x)} \theta \left(\frac{(3f+g)(\tau)}{8}+x, \frac{(f+g)(\tau)}{2} \right)\\
				&+ e^{\pi i(\frac{g(\tau)}{4}-x)} \theta \left( \frac{-(f+3g)(\tau)}{8}+x, \frac{(f+g)(\tau)}{2} \right)
			\end{split}
		\end{equation}
		
		In particular, we will choose $f(\tau), g(\tau)$ such that $f(\tau)+g(\tau)=2\tau$. 
		
		\begin{equation}
			\begin{split}
				&2\theta \left( \frac{f(\tau)-\tau}{4}+x,\tau \right)-2\theta \left( \frac{f(\tau)-\tau}{8}+\frac{x+1}{2}, \frac{\tau}{4} \right)\\
				=&e^{\pi i(\frac{f(\tau)}{4}+x)}\theta \left(\frac{f(\tau)+\tau}{4}+x, \tau \right)+ e^{\pi i(\frac{2\tau-f(\tau)}{4}-x)}\theta \left( \frac{f(\tau)-3\tau}{4}+x, \tau \right).
			\end{split}
		\end{equation}
		\item Similarly, the identity can be obtained from lemma \ref{r5}.
	\end{enumerate}
\end{proof}

The identities discussed in lemma gives new insights to the Landen transformations. 

\section{Modular equations: Generating function for quadratic numbers}

Following lemma gives the generating function for the quadratic numbers.
\begin{lemma}\label{sn1}
	\begin{equation*}
		\sum_{m=-\infty}^{\infty} q^{k_1m^2+k_2m}
		= \frac{\left(q^{\frac{k_1+2k_2}{4}},q^{\frac{k_1-2k_2}{4}},q^{\frac{k_1}{2}}: q^{\frac{k_1}{2}}\right)_{\infty}+\left(-q^{\frac{k_1+2k_2}{4}},-q^{\frac{k_1-2k_2}{4}},q^{\frac{k_1}{2}}: q^{\frac{k_1}{2}}\right)_{\infty}}{2}
	\end{equation*}
\end{lemma}
\begin{proof}
	Evaluate first identity of the corollary \ref{r4} at $x=-\tau, a= q^{k_1+2k_2+4}, b= q^{k_1-2k_2-4}.$ This gives,
	\begin{equation*}
		4 f_{(0,0,2)}(a,b,x)= 4f_{(0,0,2)}(q^{k_1+2k_2+4},q^{k_1-2k_2-4},-\tau)
		= 4\sum_{m=-\infty}^{\infty} q^{k_1m^2+k_2m}.	
	\end{equation*}
	
	and
	\begin{equation*}
		\begin{split}
			&2 f_{(0,0,1)} (a^{1/4}, b^{1/4},x/2)
			+[f_{(0,\frac{1}{2},1)}(a^{1/4}, b^{1/4},x/2)+ f_{(0,\frac{-1}{2},1)}(a^{1/4}, b^{1/4},x/2)]\\
			=&2\sum_{m=-\infty}^{\infty}a^{\frac{m(m+1)}{8}}b^{\frac{m(m-1)}{8}}e^{\pi imx}+2\sum_{m=-\infty}^{\infty}a^{\frac{m(m+1)}{8}}b^{\frac{m(m-1)}{8}}e^{\pi imx}(-1)^m,\\
			=& 2 \left[f \left(q^{\frac{k_1+2k_2}{4}}, q^{\frac{k_1-2k_2}{4}} \right) + f \left(-q^{\frac{k_1+2k_2}{4}}, -q^{\frac{k_1-2k_2}{4}} \right) \right].
		\end{split}
	\end{equation*}
	Using equation (\ref{j1}) we obtain the quadratic numbers identity:
	\begin{equation}
		\sum_{m=-\infty}^{\infty} q^{k_1m^2+k_2m}
		= \frac{\left(q^{\frac{k_1+2k_2}{4}},q^{\frac{k_1-2k_2}{4}},q^{\frac{k_1}{2}}: q^{\frac{k_1}{2}}\right)_{\infty}+\left(-q^{\frac{k_1+2k_2}{4}},-q^{\frac{k_1-2k_2}{4}},q^{\frac{k_1}{2}}: q^{\frac{k_1}{2}}\right)_{\infty}}{2}
	\end{equation}
\end{proof} 
\begin{remark}
	Different choices of $x$ and corresponding choices of $a$ and $b$ give the alternate expression for the generating functions for quadratic numbers.
\end{remark}	
The generating functions for the polygonal numbers are given in the following corollary. 
\begin{corollary}\label{R5}
	\begin{enumerate}
		\item The generating function for triangular numbers 
		\small \begin{equation*}
			\begin{split}
				&2\sum_{m=-\infty}^{\infty}q^{\frac{m(m+1)}{2}}\\
				=&\prod_{k=0}^{\infty}(1-q^{\frac{k+1}{4}})\left[\prod_{k=0}^{\infty}(1-q^{\frac{2k+3}{8}}) (1-q^{\frac{2k-1}{8}})+\prod_{k=0}^{\infty}(1+q^{\frac{2k+3}{8}}) (1+q^{\frac{2k-1}{8}}))
				\right]
			\end{split}
		\end{equation*}
		\item The generating function for square numbers
		\begin{equation*}
			2+4\sum_{m=1}^{\infty}q^{m^2}=\prod_{k=0}^{\infty}(1-q^{\frac{k+1}{2}})\left[\prod_{k=0}^{\infty}(1-q^{\frac{2k+1}{4}})^2 +\prod_{k=0}^{\infty}(1+q^{\frac{2k+1}{4}})^2
			\right]
		\end{equation*}
		
		\item The generating function for Pentagonal numbers
		\begin{equation*}
			\begin{split}
				&2\sum_{m=-\infty}^{\infty}q^{\frac{m(3m-1)}{2}}\\
				=&\prod_{k=0}^{\infty}(1-q^{\frac{3(k+1)}{4}})\left[\prod_{k=0}^{\infty}(1-q^{\frac{6k+1}{8}}) (1-q^{\frac{6k+5}{8}})+\prod_{k=0}^{\infty}(1+q^{\frac{6k+1}{8}}) (1+q^{\frac{6k+5}{8}})
				\right]
			\end{split}
		\end{equation*}
		
		\item The generating function for Hexagonal numbers
		\begin{equation*}
			2\sum_{m=-\infty}^{\infty}q^{m(2m-1)}=\prod_{k=0}^{\infty}(1-q^{2k+2})(1+q^{k})
		\end{equation*}
		
		\item The generating function for Heptagonal numbers
		\begin{equation*}
			\begin{split}
				&2\sum_{m=-\infty}^{\infty}q^{\frac{m(5m-3)}{2}}\\
				=&\prod_{k=0}^{\infty}(1-q^{\frac{5(k+1)}{4}})\left[\prod_{k=0}^{\infty}(1-q^{\frac{10k-1}{8}}) (1-q^{\frac{10k+11}{8}})+\prod_{k=0}^{\infty}(1+q^{\frac{10k-1}{8}}) (1+q^{\frac{10k+11}{8}})
				\right]
			\end{split}
		\end{equation*}
		
		\item The generating function for Octagonal numbers 
		\begin{equation*}
			\begin{split}
				&2\sum_{m=-\infty}^{\infty}q^{3m^2-2m}\\
				=&\prod_{k=0}^{\infty}(1-q^{\frac{3(k+1)}{2}})\left[\prod_{k=0}^{\infty}(1-q^{\frac{6k-1}{4}}) (1-q^{\frac{6k+7}{4}})+\prod_{k=0}^{\infty}(1+q^{\frac{6k-1}{4}}) (1+q^{\frac{6k+7}{4}})
				\right]
			\end{split}
		\end{equation*}
		\item The $m^{th}$ $r$-gonal number is defined by $\frac{(r-2)m^2+(4-r)m}{2}$. The generating function is
		
		\begin{equation*}
			\begin{split}
			&2\sum_{m=-\infty}^{\infty} q^{\frac{(r-2)m^2+(4-r)m}{2}}\\
			=& \left(q^{\frac{r-2}{4}},q^{\frac{6-r}{8}},q^{\frac{3r-10}{8}}: q^{\frac{r-2}{4}}\right)_{\infty}+\left(q^{\frac{r-2}{4}},-q^{\frac{6-r}{8}},-q^{\frac{3r-10}{8}}: q^{\frac{r-2}{4}}\right)_{\infty}
		\end{split}
		\end{equation*}
	\end{enumerate}
\end{corollary}
\begin{remark}
	\begin{enumerate}
		\item A fifth power of the third identity of the corollary \ref{R5} is
		\begin{equation*}
			\left(\sum_{m=-\infty}^{\infty}q^{\frac{m(3m-1)}{2}}\right)^5=\sum_{n=0}^{\infty} p(n) q^n
		\end{equation*}
		where $p(n)$ is the number of ways by which a non negative integer can be expressed as a sum of five pentagonal numbers, and 
		
		\begin{equation*}
			\sum_{n=0}^{\infty} p(n) q^n= 1+ 5q+15q^2+30q^3+45q^4+56q^5+65q^6+85q^7+115q^8+150q^9+\cdots .
		\end{equation*}
		
		\item Lemma \ref{sn1} also gives a generating function for the numbers of the form $k_1m^2+k_2m+k_3$.
	\end{enumerate}
\end{remark}

\section{Conclusion} A new technique to obtain the functional relations between Ramanujan theta functions has been developed. This idea can be extended to higher dimensional DFT's to obtain some significant identities between Jacobi theta functions.

	\bibliographystyle{sn-jnl}

	
	
\end{document}